\documentclass{amsart}
\usepackage{amsmath,amssymb,amsthm}
\usepackage{hyperref}
\usepackage{xcolor}

\theoremstyle{plain}
\newtheorem{theorem}{Theorem}[section]

\newtheorem{corollary}{Corollary}[section]
\newtheorem{proposition}{Proposition}[section]

\theoremstyle{definition}
\newtheorem{definition}{Definition}[section]

\newtheorem*{conjecture*}{Conjecture}
\newtheorem{example}{Example}[section]

\begin{document}

\title[Fuzzy Sets and the Ershov Hierarchy]{Approximating Approximate Reasoning: Fuzzy Sets and the Ershov Hierarchy\thanks{The work was supported by Nazarbayev University Faculty Development
Competitive Research Grants 021220FD3851. Bazhenov and Ospichev were supported by Mathematical Center in Akademgorodok under agreement No. 075-15-2019-1613 with the Ministry of Science and Higher Education of the Russian Federation. San Mauro has been partially supported by  the Austrian Science Fund FWF, project M~2461.}}
%

%
\author{Nikolay Bazhenov}
\address{Sobolev Institute of Mathematics, 4 Acad. Koptyug Ave., Novosibirsk, 630090, Russia}
\email{bazhenov@math.nsc.ru}

\author{Manat Mustafa}
\address{Department of Mathematics, School of Sciences and Humanities, Nazarbayev University, 53 Qabanbaybatyr Avenue, Nur-Sultan, 010000, Kazakhstan}
\email{manat.mustafa@nu.edu.kz} 

\author{Sergei Ospichev}
\address{Sobolev Institute of Mathematics, 4 Acad. Koptyug Ave., Novosibirsk, 630090, Russia}
\email{ospichev@math.nsc.ru}

\author{Luca San Mauro}
\address{Institute of Discrete Mathematics and Geometry, Vienna University of Technology}
\email{luca.sanmauro@gmail.com}

%
%

\begin{abstract}
Computability theorists have introduced multiple hierarchies to measure the complexity of sets of natural numbers. The Kleene Hierarchy classifies sets according to the first-order complexity of their defining formulas. The Ershov Hierarchy classifies $\Delta^0_2$ sets with respect to the number of mistakes that are needed to approximate them. Biacino and Gerla extended the Kleene Hierarchy to the realm of fuzzy sets, whose membership functions range in a complete lattice $L$ (e.g., the real interval $[0; 1]_\mathbb{R}$). In this paper, we combine the Ershov Hierarchy and fuzzy set theory, by introducing and investigating the Fuzzy Ershov Hierarchy. In particular, we focus on the fuzzy $n$-c.e. sets which form the finite levels of this hierarchy. Intuitively, a fuzzy set is $n$-c.e.\ if its membership function can be approximated by changing monotonicity at most $n-1$ times. We prove that the Fuzzy Ershov Hierarchy does not collapse; that, in analogy with the classical case, each fuzzy $n$-c.e.\ set can be represented as a Boolean combination of fuzzy c.e.\ sets; but that, contrary to the classical case, the Fuzzy Ershov Hierarchy does not exhaust the class of all $\Delta^0_2$ fuzzy sets.

\end{abstract}

\keywords{Fuzzy set, computability theory, $n$-computably enumerable set, Ershov hierarchy.}

\maketitle              
%


\section{Introduction}

\emph{Crisp properties} on a given domain $\mathcal{D}$~--- i.e., properties whose membership functions range in the set $\{0,1\}$~--- can be naturally identified with subsets of $\mathcal{D}$. By adopting this perspective, one may regard classical computability theory as the study of the complexity of crisp properties on the set $\omega$ of the natural numbers: e.g., ``being even''  and ``being the code of a Turing machine which halts on a blank tape'' are examples of, respectively, a decidable crisp property and an undecidable one. 

Computability theorists have introduced multiple hierarchies to measure the complexity of crisp objects. Two such hierarchies will be relevant for the present paper. The Kleene Hierarchy classifies subsets of $\omega$ according to the first-or\-der complexity of their defining formulas within arithmetic. The Ershov Hierarchy concentrates on an important initial segment of the Kleene Hierarchy, that of  $\Delta^0_2$ sets (which coincide with the sets that are computable in the limit), by classifying such sets with respect to the
number of mistakes that are needed to approximate them.

Fuzzy sets, introduced by Zadeh~\cite{Zadeh-65} and later developed into a broad area of research, allow to mathematically study \emph{graded properties}, such as those properties with blurry boundaries, and to extend the scope of logic to \emph{approximate reasoning}. 

It is natural to ask how to introduce  computability theory within fuzzy mathematics. 
A first approach is to define fuzzy algorithms (as in, e.g., \cite{Zadeh-68,Santos-70,wiedermann2002,bedregal2006}), and then rebuild computability theory by permitting fuzzy computations. A parallel approach is to maintain ordinary Turing machines and just adopt them to calibrate the complexity of fuzzy sets. After all, well-established computa\-bi\-li\-ty-the\-o\-re\-tic hierarchies could be extended to the realm of fuzzy objects. This is the case for the Kleene Hierarchy, which has been extended to fuzzy sets by Biacino and Gerla~\cite{BG-89}, see also \cite{Gerla-book,Harkleroad-84,Harkleroad-88}.
 
In this paper, we introduce and investigate the Fuzzy Ershov Hierarchy. That is, we focus on the complexity of approximating fuzzy objects, which belong to the class $\Delta^0_2$. The key idea for evaluating this complexity is that of a \emph{mind change}, which is borrowed from Ershov~\cite{Ershov-I,Ershov-II,Ershov-III}. In the classical setting, an approximation to a set $A$ changes its mind on a given input $x$ by switching its guess on whether $x$ belongs to $A$ or not. Moving to fuzzy sets, mind changes will be formalized by changes in the monotonicity of approximating functions.  We will prove that, by allowing more and more  mind changes, we will be able  to capture larger and larger sub-classes of $\Delta^0_2$ fuzzy sets. In particular, it will follow that there are fuzzy sets which cannot be approximated only from above or below, but they require  approximations which 
 oscillate ``up and down'' on the membership degree of $x$, for some inputs $x$. 
 
The paper is arranged as follows. In Section~\ref{sect:prelim}, we recall  preliminaries concerning fuzzy sets, effective reals, and the Ershov Hierarchy. In Section~\ref{sect:results}, we introduce the Fuzzy Ershov Hierarchy, and we prove the main results of the paper. First, the hierarchy does not collapse (Proposition~\ref{prop:non-collapse}). Second, in analogy with the classical case, sets lying at the so-called finite levels of the Fuzzy Ershov Hierarchy can be represented as  Boolean combinations of fuzzy sets belonging to the first level, i.e. fuzzy c.e.\ sets (Theorem~\ref{theo:Boolean-fuzzy}). Third, contrary to the classical case, the Fuzzy Ershov Hierarchy does not exhaust the class of all $\Delta^0_2$ fuzzy sets (Proposition~\ref{prop:not-enough}). In the last section, we conclude by briefly discussing two natural ways of broadening the Fuzzy Ershov Hierarchy.


\section{Preliminaries} \label{sect:prelim}

We assume that the reader is familiar with the basic notions of computability theory. For the background, we refer to the monographs~\cite{Rogers-Book,Soare-16}. The preliminaries on fuzzy sets mainly follow~\cite{Gerla-book}.

As usual, one fixes an effective bijection $\nu\colon \mathbb{Q}\to \omega$. This convention allows to transfer familiar computability-theoretic notions to the rationals: for example, a crisp set $X\subseteq \mathbb{Q}$ is computable iff its image $\nu(X)$ is a computable subset of $\omega$.

Note that, as is custom in computability theory, this paper uses the term \emph{computably enumerable} (or \emph{c.e.}) in place of \emph{recursively enumerable}. For a set $X$, by $|X|$ we denote the cardinality of $X$.

\subsection{Fuzzy Subsets}

Let $\mathcal{L}$ be a complete lattice. A \emph{fuzzy subset} (or an $\mathcal{L}$-subset) of $\omega$ is an arbitrary function $A \colon \omega \to \mathcal{L}$. 
In this paper, for the sake of simplicity, we consider the case when $\mathcal{L}$ is equal to the real interval $[0;1]_{\mathbb{R}}$. A fuzzy subset $A$ is \emph{crisp} if $A(x) \in \{ 0,1\}$ for all $x\in\omega$. 

As mentioned in  the introduction, a fundamental tool for classifying the complexity of crisp subsets of $\omega$ is provided by the Kleene Arithmetical Hierarchy~\cite{Kleene-43} (see, e.g., Chapter~4 in~\cite{Soare-16} for a detailed discussion). Biacino and Gerla~\cite{BG-89} extended the Kleene Hierarchy to fuzzy subsets. 
In our paper, we work only with fuzzy subsets belonging to the levels $\Sigma^0_1$, $\Pi^0_1$, and $\Delta^0_2$ of the Kleene Hierarchy. Hence, we give formal definitions only for these levels. For more details, the reader is referred to~\cite{BG-89} and \S\,11.5 in~\cite{Gerla-book}. 
By $[0;1]_{\mathbb{Q}}$ we denote the set of all rational numbers $q$ such that $0\leq q \leq 1$.

\begin{definition}[\cite{BG-89}, see also \S\,11.2 in~{\cite{Gerla-book}}] \label{def:fuzzy_c_e}
	A fuzzy set $A$ is \emph{computably enumerable} (or belongs to the class $\Sigma^0_1$) if there is a computable function $f\colon \omega\times\omega \to [0,1]_{\mathbb{Q}}$ such that, for all $x\in\omega$, we have:
	\begin{itemize}		
		\item[(a)] $\lim_{s\to \infty}f(x,s)=A(x)$;
		
		\item[(b)] $(\forall s)(f(x,s+1)\geq f(x,s))$.
	\end{itemize}
	We say that such function $f$ is a \emph{$\Sigma^0_1$-approximation} of the fuzzy set $A$.
\end{definition}

Note that without loss of generality, one can always assume that in the definition above, $f(x,0)$ equals $0$. Hence, fuzzy c.e.\ sets may intuitively be regarded as fuzzy sets which can be approximated ``from below'', in the sense that approximations to fuzzy c.e.\ sets can only increase over time.

If $A$ and $B$ are fuzzy sets, then one can define set-theoretic operations on them:
\begin{itemize}
	\item Union: $(A\cup B)(x) = \max\{ A(x), B(x)\}$.
	
	\item Intersection: $(A\cap B)(x) = \min \{A(x), B(x)\}$.
	
	\item Complement: $\overline{A}(x) = 1 - A(x)$.
\end{itemize}

A fuzzy set $A$ is \emph{co-computably enumerable} (or belongs to the class $\Pi^0_1$) if its complement $\overline{A}$ is c.e. Equivalently (see Theorem~5.2 in~{\cite[Chap.~11]{Gerla-book}}), $A$ is co-c.e. if and only if there is a computable function $f\colon \omega^2 \to [0;1]_{\mathbb{Q}}$ such that, for all $x\in\omega$, we have:
\begin{itemize}
	\item[(\emph{a})] $\lim_{s} f(x,s)=A(x)$;
	\item[(\emph{b}$'$)] $(\forall s)(f(x,s+1)\leq f(x,s))$.
\end{itemize}
In the $\Pi^0_1$ case, we may assume that $f(x,0)=1$, for all $x$. So, fuzzy co-c.e.\ sets may be regarded as fuzzy sets which can be approximated ``from above''.

Finally, the main object of study of this paper are $\Delta^0_2$ fuzzy sets. A fuzzy set $A$ \emph{belongs to the class $\Delta^0_2$} if $A$ lies in both classes $\Sigma^0_2$ and $\Pi^0_2$ of the Kleene Hierarchy. In this paper, we  adopt the following equivalent definition (see Proposition~5.4 in {\cite[Chap.~11]{Gerla-book}}). 

\begin{definition}
A fuzzy set $A$ is $\Delta^0_2$ if and only if there is a computable function $f\colon \omega^2 \to [0;1]_{\mathbb{Q}}$ such that $\lim_{s}f(x,s)=A(x)$, for all $x\in\omega$. We call such function $f(x,s)$ a \emph{$\Delta^0_2$-approximation} of the fuzzy set $A$.
\end{definition}

\subsection{Effective Reals}

Here we briefly discuss some simple results which connect fuzzy subsets of $\omega$ with effectively approximable reals. 
We refer to Chapter~5 in~\cite{DH-book} for the detailed background.

We consider reals $\alpha \in [0;1]_{\mathbb{R}}$. A real $\alpha$ is \emph{left-c.e.} if the set $\{ q\in\mathbb{Q}\,\colon q < \alpha\}$ is c.e. A real $\alpha$ is \emph{right-c.e.} if the set $\{ q\in\mathbb{Q} \,\colon q > \alpha\}$ is c.e.
By working with the definitions, it is not hard to prove the following result.

\begin{proposition}\label{prop:left-right}
	 Let $A$ be a fuzzy subset of $\omega$. 
	 \begin{itemize} 
	 	\item[(1)]  $A$ is c.e. if and only if the reals $A(k)$, $k\in\omega$, are uniformly left-c.e., i.e. the set $\{ (k,q) \in \omega\times \mathbb{Q} \,\colon q < \alpha_k\}$ is c.e. 
	 	
	 	\item[(2)] $A$ is co-c.e. if and only if the reals $A(k)$, $k\in\omega$, are uniformly right-c.e., i.e. the set $\{ (k,q) \in \omega\times \mathbb{Q} \,\colon q > \alpha_k\}$ is c.e. 
	\end{itemize}
\end{proposition}

A real $\alpha$ is $\Delta^0_2$ if there is a computable sequence $(q_s)_{s\in\omega}$ of rationals such that $\alpha = \lim_s q_s$ (see, e.g., Theorem~5.1.3 in~\cite{DH-book}). We also observe the following, which is an immediate consequence of the definitions:

\begin{proposition}\label{remark:Delta_2}
	A fuzzy set $A$ is $\Delta^0_2$ if and only if the reals $A(k)$, $k\in\omega$, are uniformly $\Delta^0_2$, i.e., there is a computable sequence $(q_{k,s})_{k,s\in\omega}$ of rationals such that $A(k) = \lim_s q_{k,s}$, for all $k$.
\end{proposition}


\subsection{The Classical Ershov Hierarchy}

We give few preliminaries on the Ershov Hierarchy~\cite{Ershov-I,Ershov-II,Ershov-III}; to distinguish it from the fuzzy analogue introduced below, we refer to this hierarchy as the Classical Ershov Hierarchy. For the sake of simplicity, here we discuss only the finite levels of the hierarchy (these finite levels are also called \emph{Difference Hierarchy} in the literature).  In this section, all subsets of $\omega$ are crisp.

\begin{definition}\label{def:Ershov}
	Let $n\geq 1$. A set $A\subseteq \omega$ is \emph{$n$-computably enumerable} (or \emph{$n$-c.e.}, or belongs to the class $\Sigma^{-1}_n$) if there is a computable function $f\colon \omega\times \omega \to \{ 0,1\}$ such that for all $x\in\omega$, we have
	\begin{itemize}
		\item $\lim_s f(x,s) = A(x)$;
		
		\item $f(x,0) = 0$;
		
		\item $|\{ s\,\colon f(x,s) \neq f(x,s+1)\}| \leq n$.
	\end{itemize}
	A set $A$ is \emph{co-$n$-computably enumerable} (or \emph{co-$n$-c.e.}, or belongs to the class $\Pi^{-1}_n$) if its complement $\overline{A}$ is $n$-c.e.
\end{definition}

Historically, the notion of $n$-c.e. sets was introduced by Putnam~\cite{Putnam} and Gold~\cite{Gold}. Note that
\[
	\Sigma^{-1}_n \cup \Pi^{-1}_n \subseteq \Sigma^{-1}_{n+1} \cap \Pi^{-1}_{n+1}.
\]
Ershov~\cite{Ershov-I} proved that for each $n\geq 1$, there exists an $n$-c.e. set $S_n$ such that every $n$-c.e. set $A$ is many-one reducible to $S_n$. In addition, $S_n$ does not belong to $\Pi^{-1}_n$. In particular, this implies that the Classical Ershov Hierarchy does not collapse. 

Sets from the class $\Sigma^{-1}_n$ can be represented as Boolean combinations of c.e. sets:

\begin{theorem}[{\cite{Ershov-I}}] \label{theo:classical-Boolean}
	Let $k$ be a natural number. A set $A\subseteq \omega$ is $(2k+1)$-c.e. if and only if there are c.e. sets $W_1,W_2,\dots,W_{2k-1},W_{2k},W_{2k+1}$ such that
	\[
		A = (W_1 \setminus W_2) \cup (W_3 \setminus W_4) \cup \dots \cup (W_{2k-1}\setminus W_{2k}) \cup W_{2k+1}.
	\]
	A set $A$ is $(2k+2)$-c.e. if and only if there are c.e. sets $W_1,W_2,\dots,W_{2k+1},W_{2k+2}$ such that
	\[
		A = (W_1 \setminus W_2) \cup (W_3 \setminus W_4) \cup \dots \cup (W_{2k-1}\setminus W_{2k}) \cup (W_{2k+1}\setminus W_{2k+2}).
	\]
\end{theorem}

We refer the reader to the survey~\cite{SYY-09} for more details on the Classical Ershov Hierarchy.


\section{Fuzzy Ershov Hierarchy} \label{sect:results}

In this section, we extend the classical Difference Hierarchy to the class of fuzzy subsets of $\omega$ (see Definition~\ref{def:fuzzy-n-c.e.} below). We establish some initial properties of this hierarchy: the hierarchy does not collapse (Subsection~\ref{subsect:not-collapse}); it is connected to the Boolean combinations of fuzzy c.e. sets (Subsection~\ref{subsect:Boolean}); the introduced levels of the hierarchy do not exhaust all fuzzy $\Delta^0_2$ sets (Subsection~\ref{subsect:not_enough}).

We begin by illustrating the intuition behind Definition~\ref{def:fuzzy-n-c.e.} with the following example. A fuzzy $\Delta^0_2$ set $A$ is called $3$-computably enumerable if it possesses a $\Delta^0_2$-approximation $f(x,s)$, which ``changes its mind'' \emph{at most two times}. So, for an element $x\in\omega$, the worst case behavior looks like this:
\begin{itemize}
	\item First, our approximation (non-strictly) increases~--- i.e. there is a stage $s_1$ such that $f(x,s) \leq f(x,s+1)$, for all $s< s_1$.
	
	\item Second, the approximation starts to decrease until some stage $s_2>s_1$: $f(x, s_1) > f(x, s_1 +1)$ and $f(x,s) \geq f(x,s+1)$, for $s_1 < s < s_2$.
	
	\item Then the final change of mind happens: the approximation will forever increase~--- $f(x,s_2) < f(x,s_2+1)$ and $f(x,s) \leq f(x,s+1)$ for all $s> s_2$.
\end{itemize}
In order to make this idea formal, we introduce \emph{mind change functions}, which ``track down'' the described mind changes.

\begin{definition}
	Let $f$ be a total function from $\omega \times \omega$ to $[0;1]_{\mathbb{Q}}$. Its \emph{$\Sigma$-mind change function} $m^{f}_{\Sigma} \colon \omega\times \omega \to \{ -1, 1\}$ is defined as follows.
	\begin{enumerate}
		\item $m^f_{\Sigma}(x,0) = 1$.
		
		\item Suppose that $m^f_{\Sigma}(x,s) = 1$.
		\begin{itemize} 
			\item If $f(x,s) \leq f(x,s+1)$, then $m^f_{\Sigma}(x,s+1) = 1$.
			
			\item If $f(x,s) > f(x,s+1)$, then $m^f_{\Sigma}(x,s+1) = -1$.
		\end{itemize}
		
		\item Suppose that $m^f_{\Sigma}(x,s) = -1$.
		\begin{itemize} 
			\item If $f(x,s) \geq f(x,s+1)$, then $m^f_{\Sigma}(x,s+1) = -1$.
			
			\item If $f(x,s) < f(x,s+1)$, then $m^f_{\Sigma}(x,s+1) = 1$.
		\end{itemize}		
	\end{enumerate}
	The \emph{$\Pi$-mind change function}  $m^{f}_{\Pi}(x,s)$ is defined similarly to $m^f_{\Sigma}$, with the following key modification: we put $m^f_{\Pi}(x,0) = -1$.
\end{definition}

Notice the following: if a function $f$ is computable, then both $m_{\Sigma}^f$ and $m_{\Pi}^f$ are also computable. Now we are ready to give the main definition.

\begin{definition}\label{def:fuzzy-n-c.e.}
	Let $n$ be a non-zero natural number. A fuzzy set $A$ is \emph{$n$-compu\-ta\-bly enumerable} if there is a computable function $f\colon \omega\times\omega \to [0,1]_{\mathbb{Q}}$ such that for all $x\in\omega$, we have:
	\begin{itemize}		
		\item $\lim_{s} f(x,s)=A(x)$;
		
		\item $f(x,0) = 0$;
		
		\item $|\{ s\in \omega \,\colon m^f_{\Sigma}(x,s+1) \neq  m^f_{\Sigma}(x,s)\}| \leq n-1$.
	\end{itemize}	
	A fuzzy set $A$ is \emph{co-$n$-computably enumerable} if its complement $\overline{A}$ is $n$-c.e.
\end{definition}

Note that 1-c.e. fuzzy sets are precisely the c.e. sets from Definition~\ref{def:fuzzy_c_e}. In addition, the following fact is immediate.

\begin{proposition}
	Let $n\geq 1$. A fuzzy set $A$ is co-$n$-c.e. if and only if there is a computable function $f\colon \omega\times\omega \to [0,1]_{\mathbb{Q}}$ such that for all $x\in\omega$, we have:
	\begin{itemize}		
		\item $\lim_{s} f(x,s)=A(x)$;
		
		\item $f(x,0) = 1$;
		
		\item $|\{ s\in \omega \,\colon m^f_{\Pi}(x,s+1) \neq  m^f_{\Pi}(x,s)\}| \leq n-1$.
	\end{itemize}	
\end{proposition}

\subsection{The Hierarchy Does Not Collapse} \label{subsect:not-collapse}

In order to show the non-collapse of the hierarchy, it is sufficient to prove the following:

\begin{proposition}\label{prop:non-collapse}
	Let $A$ be a crisp subset of $\omega$. Then $A$ is $n$-c.e. in the Classical Ershov Hierarchy if and only if $A$ is $n$-c.e. in the Fuzzy Ershov Hierarchy. A similar fact is true for co-$n$-c.e. sets.
\end{proposition}

Indeed, since the Classical Difference Hierarchy does not collapse, Proposition~\ref{prop:non-collapse} implies that our hierarchy is also non-collapsing.

\begin{proof}[of Proposition~\ref{prop:non-collapse}]
	$(\Rightarrow)$. Suppose that $A$ is $n$-c.e. in the classical sense. We fix a computable function $f\colon \omega^2 \to \{ 0,1\}$ satisfying the conditions from Definition~\ref{def:Ershov}. It is clear that $f(x,s)$ is a $\Delta^0_2$-approximation of $A$, treated as a fuzzy $\Delta^0_2$ set.
	
	For an element $x\in\omega$, consider all stages $s_1 < s_2 <\dots < s_k$ (note that $k\leq n$) such that $f(x,s_i) \neq f(x,s_i+1)$. A straightforward analysis shows the following:
	\begin{itemize}
		\item If $s \leq s_1$, then $f(x,s) = 0$ and $m^{f}_{\Sigma}(x,s) = 1$.
		
		\item If $s_{2\ell+1} +1 \leq s \leq s_{2\ell+2}$, then $f(x,s) = 1$ and $m^{f}_{\Sigma}(x,s) = 1$.
		
		\item If $s_{2\ell+2} +1 \leq s \leq s_{2\ell+3}$, then $f(x,s) = 0$ and $m^{f}_{\Sigma}(x,s) = -1$.
	\end{itemize}
	This implies that $|\{ s\,\colon m^f_{\Sigma}(x,s+1) \neq m^f_{\Sigma}(x,s)\}| \leq k-1 \leq n-1$. We deduce that the approximation $f$ witnesses that the set $A$ is fuzzy $n$-c.e.
	
	$(\Leftarrow)$. Suppose that a crisp $A$ is fuzzy $n$-c.e. Fix a $\Delta^0_2$-approximation $f\colon \omega^2 \to [0;1]_{\mathbb{Q}}$ satisfying  Definition~\ref{def:fuzzy-n-c.e.}. We define a new approximation
	\[
		g(x,s) = \begin{cases}
			1, & \text{if } f(x,s) > 1/2,\\
			0, & \text{if } f(x,s) \leq 1/2.
		\end{cases}
	\]
	Since $A$ is crisp, it is clear that $A(x) = \lim_s g(x,s)$. Notice that $g(x,0) = f(x,0) = 0$.
	
	For an element $x\in\omega$, consider all stages $s_1 < s_2 < \dots < s_k$ such that $g(x,s_i) \neq g(x,s_i+1)$. For $i\leq k$, one can show the following: 
	\begin{itemize}
		\item If $i=2\ell+1$, then $f(x,s_i) \leq 1/2$, $f(x,s_i +1) > 1/2$, and $m^{f}_{\Sigma}(x,s_i + 1) = 1$.
		
		\item If $i=2\ell+2$, then $f(x,s_i) > 1/2$, $f(x,s_i +1) \leq 1/2$, and $m^{f}_{\Sigma}(x,s_i + 1) = -1$.
	\end{itemize}
	In turn, this implies $k-1 \leq |\{ s\,\colon m^f_{\Sigma}(x,s+1) \neq m^f_{\Sigma}(x,s)\}| \leq n-1$. Hence, $k\leq n$, and the function $g(x,s)$ witnesses that the set $A$ is $n$-c.e. in the classical sense.
	\qed
\end{proof}

\subsection{Boolean Combinations of Fuzzy C.E. Sets} \label{subsect:Boolean}

We show that similarly to the Classical Ershov Hierarchy (Theorem~\ref{theo:classical-Boolean}), $n$-c.e. fuzzy sets admit natural presentations via Boolean combinations of c.e. sets.

\begin{theorem}\label{theo:Boolean-fuzzy}
	Let $n\in \{2 k+1, 2k+2\}$. A fuzzy set $C$ is $n$-c.e. if and only if there are fuzzy c.e. sets $A_1,B_1,A_2,B_2,\dots,A_{k+1},B_{k+1}$ such that:
	\begin{itemize}
		\item $C = (A_1 \cap \overline{B_1}) \cup (A_2 \cap \overline{B_2}) \cup \dots \cup (A_{k+1} \cap \overline{B_{k+1}})$;
		
		\item if $n = 2k+1$, then $B_{k+1} = \emptyset$.
	\end{itemize}
\end{theorem}
\begin{proof}
		$(\Rightarrow)$. Let $f(x,s)$ be a $\Delta^0_2$-approximation which witnesses the fact that $C$ is $n$-c.e. We define the desired fuzzy c.e. sets $A_i$ and $B_i$ via their $\Sigma^0_1$-approxi\-ma\-ti\-ons $h_{A_i}$ and $h_{B_i}$ (in the sense of Definition~\ref{def:fuzzy_c_e}), respectively.
		
		The intuition behind these c.e. sets is as follows. For an element $x\in\omega$, we split $\omega$ into disjoint intervals: $[0;a_0)$, $[a_0;b_0)$, $[b_0;a_1)$, $[a_1;b_1)$, etc. Our function $f(x,\cdot)$ (non-strictly) increases on the intervals $[0;a_0)$, $[b_0;a_1)$, $[b_1;a_2)$, etc. The function decreases on the rest of the intervals.
		\begin{itemize}
			\item The approximation $h_{A_1}$ of the set $A_1$ looks like this: it copies $f(x,\cdot)$ on the interval $[0;a_0)$, and then stabilizes, i.e. $h_{A_1}(x,s) = h_{A_1}(x,a_0-1)$ for all $s\geq a_0$.
			
			\item The function $h_{B_1}(x,\cdot)$ equals zero on $[0;a_0)$. Then it copies $1-f(x,\cdot)$ on the interval $[a_0;b_0)$. After that, $h_{B_1}(x,\cdot)$ equals \emph{one}.
			
			\item The function $h_{A_2}$ equals zero on $[0;b_0)$. Then it copies $f(x,\cdot)$ on the interval $[b_0;a_1)$; after that $h_{A_2}$ stabilizes. Et cetera.
		\end{itemize}		
		Formally speaking, for a non-zero $i \leq k+1$, we define:
		\begin{gather*}
			h_{A_i}(x,s) = \begin{cases}
				0,  \quad\text{if }  |\{ t\leq s \,\colon m^{f}_{\Sigma}(x,t+1) \neq m^{f}_{\Sigma}(x,t)\}| < 2i-2,\\
				f(x,s),  \quad\text{if } |\{ t\leq s \,\colon m^{f}_{\Sigma}(x,t+1) \neq m^{f}_{\Sigma}(x,t)\}| = 2i-2,\\
				h_{A_i}(x,s-1),  \quad\text{otherwise};
			\end{cases}\\
			h_{B_i}(x,s) = \begin{cases}
				0,  \quad\text{if } |\{ t\leq s \,\colon m^{f}_{\Sigma}(x,t+1) \neq m^{f}_{\Sigma}(x,t)\}| < 2i-1,\\
				1-f(x,s),  \quad\text{if } |\{ t\leq s \,\colon m^{f}_{\Sigma}(x,t+1) \neq m^{f}_{\Sigma}(x,t)\}| = 2i-1,\\
				1,  \quad\text{otherwise}.
			\end{cases}
		\end{gather*}
		It is not hard to see that these approximations induce fuzzy c.e. sets. In addition, if $n=2k+1$, then $B_{k+1}(x) = 0$ for all $x$.
		
		Let $D$ be the fuzzy set $(A_1 \cap \overline{B_1}) \cup \dots \cup (A_{k+1} \cap \overline{B_{k+1}})$. We consider its natural $\Delta^0_2$-approximation
		\begin{equation}\label{equ:max-min}
			h_D(x,s) = \max\{ \min\{ h_{A_i}(x,s), 1-h_{B_i}(x,s)\} \,\colon 1\leq i \leq k+1\}.
		\end{equation}
		Consider the value $v^{\ast} = |\{ t\in\omega \,\colon m^{f}_{\Sigma}(x,t+1) \neq m^{f}_{\Sigma}(x,t)\}|$. 
		
		If $v^{\ast} = 2i-2$, then there is a stage $s^{\ast}$ such that for all $s\geq s^{\ast}$, we have
		$h_D(x,s) = h_{A_i}(x,s) = f(x,s)$. This implies that $D(x) = C(x)$.
		
		If $v^{\ast} = 2i-1$, then consider the highest index $s_0$ such that $h_{A_i}(x,s_0) = f(x,s_0)$. Then for every $s\geq s_0 +1$, we have $h_{A_i}(x,s) = f(x,s_0)$, $h_{B_i}(x,s) = 1 - f(x,s)$, and 
		\[
			h_{D}(x,s) = \min (h_{A_i}(x,s), 1-h_{B_i}(x,s)) = \min(f(x,s_0), f(x,s)) = f(x,s).
		\]
		Again, $D(x) = C(x)$. We deduce that the fuzzy sets $C$ and $D$ are equal.
		\smallskip
		
		$(\Leftarrow)$. Let $D$ be a fuzzy $\Delta^0_2$ set defined via the approximation $h_D$ from~(\ref{equ:max-min}). We prove that \emph{this} approximation $h_D$ witnesses the fact that $D$ is $n$-c.e.
		
		First, we note the following easy observation (it follows from computable enumerability of fuzzy sets $A_i$ and $B_i$): 
		\begin{itemize}
			\item[($\ast$)] If $1-h_{B_i}(x,s_0) < h_{A_i}(x,s_0)$ for some $s_0$, then we have $1-h_{B_i}(x,s) < h_{A_i}(x,s)$ \emph{for all} $s\geq s_0$.
		\end{itemize}
		An informal intuition concerning further proof is as follows. Every (approximation of the) real $(A_i \cap\overline{B}_i)(x)$ can be treated as a ``hill'': first we go up, copying the function $h_{A_i}(x,\cdot)$. When we see the inequality $1-h_{B_i}(x,s_0) < h_{A_i}(x,s_0)$, we can only go down. Coming back to the whole picture of $h_D$: whenever the mind-change function $m^{h_D}_{\Sigma}(x,\cdot)$ changes from $+1$ to $-1$, it happens because we encountered the \emph{top} of one of the ``hills''.
		
		At a stage $s$, consider the following sets: $X_s = \{  i\,\colon 1-h_{B_i}(x,s) < h_{A_i}(x,s)\}$ and $Y_s = \{1,2,\dots,k+1\} \setminus X_s$. Observation $(\ast)$ implies that $X_{s} \subseteq X_{s+1}$ for every $s$. In addition, $X_0=\emptyset$.
		
		It is not hard to deduce the following equation:
		\[
			h_D(x,s) = \max \{ \max\{ 1-h_{B_i}(x,s)\,\colon i \in X_s\}, \max\{ h_{A_i}(x,s)\,\colon i \in Y_s\} \}.
		\]
		Note that for a fixed non-empty set $Z$, the function $\max\{ 1-h_{B_i}(x,s)\,\colon i \in Z\}$ is non-increasing, and  $\max\{ h_{A_i}(x,s)\,\colon i \in Z\}$ is non-decreasing.
		
		Suppose that $m^{h_D}_{\Sigma}(x,s) = 1$ and $m^{h_D}_{\Sigma}(x,s+1) = -1$. Choose the greatest $s'< s$ such that either $s'=0$, or $s'>0$ and $m^{h_D}_{\Sigma}(x,s')=-1$. Towards a contradiction, assume that $X_{s+1}=X_{s'}$. 
		
		Then on one hand, we have 
		\[
			h_D(x,s) = \max\{ 1-h_{B_i}(x,s)\,\colon i \in X_{s'}\} > \max\{ h_{A_i}(x,s)\,\colon i \in Y_{s'}\}.
		\]
		Indeed, if $h_D(x,s)$ equals $\max\{ h_{A_i}(x,s)\,\colon i \in Y_{s'}\}$, then we would have 
		\[
			h_D(x,s+1)  = \max\{ h_{A_i}(x,s+1) \,\colon i \in Y_{s'}\} \geq h_D(x,s),
		\]
		which contradicts the fact that $m^{h_D}_{\Sigma}(x,s+1) = -1$.		
		
		On the other hand, every $t$ such that $s'< t \leq s$ satisfies 
		\[
			h_D(x,t)  = \max\{ h_{A_i}(x,t)\,\colon i \in Y_{s'} \}.
		\] 
		We obtain a contradiction. Therefore, $X_{s+1}\neq X_{s'}$.
		
		We deduce that for each stage $s$ with $m^{h_D}_{\Sigma}(x,s) = 1$ and $m^{h_D}_{\Sigma}(x,s+1) = -1$, at least one new element is added to the growing set $X = \bigcup_{t\in \omega} X_t$. 
		
		Suppose that $n=2k+2$. Then one can show that $|X| \leq k+1$. We notice the following: if $|X|$ is less than $k+1$, then the number of monotonicity breaks (of the function $m^{h_D}_{\Sigma}(x,\cdot)$) will be strictly less than the corresponding number for the case $|X|=k+1$. Hence, one can consider only the case when $|X|=k+1$.
		
		If $|X|=k+1$, then there is a stage $s^{\ast}$ such that for all $s\geq s^{\ast}$, we have $h_D(x,s) = \max\{ 1-h_{B_i}(x,s)\,\colon i \in X\}$, and this function can only decrease. A not difficult combinatorial argument shows that $|\{ s\in\omega \,\colon m^{h_D}_{\Sigma}(x,s+1) \neq m^{h_D}_{\Sigma}(x,s)\}| \leq 2k+1$.
		
		If $n=2k+1$, then $|X|\leq k$. An argument similar to the one above shows that one can consider only the case when $|X|$ equals $k$.
		
		If $|X|=k$, then there is a stage $s^{\ast}$ such that for $s\geq s^{\ast}$, we have
		$h_D(x,s) = \max\{ \max\{ 1-h_{B_i}(x,s)\,\colon i \in X \},  h_{A_{k+1}}(x,s)\}$.
		One can show that in this case, $|\{ s\,\colon m^{h_D}_{\Sigma}(x,s+1) \neq m^{h_D}_{\Sigma}(x,s)\}| \leq 2k$. Theorem~\ref{theo:Boolean-fuzzy} is proved.
		\qed
\end{proof}

\begin{corollary}
	Every finite Boolean combination of fuzzy c.e. sets is an $n$-c.e. set, for some $n\geq 1$.
\end{corollary}

\subsection{The Introduced Hierarchy Is Not Enough} \label{subsect:not_enough}

Here we show that the introduced levels of the Fuzzy Ershov Hierarchy do not exhaust the class of all $\Delta^0_2$ fuzzy subsets of $\omega$.

\begin{proposition}\label{prop:not-enough}
	There exists a $\Delta^0_2$ fuzzy set $A$ such that for any $\Delta^0_2$-approxi\-ma\-ti\-on $f(x,s)$ of $A$, the sequence $(m^f_{\Sigma}(0,s))_{s\in\omega}$ diverges when $s$ tends to infinity. In particular, $A$ is not $n$-c.e., for all $n\geq 1$.
\end{proposition}
\begin{proof}
	Choose an arbitrary $\Delta^0_2$ real $\alpha$, which is not left-c.e. and not right-c.e. (see, e.g., Theorem~5.1.10 in~\cite{DH-book} for an example of such real). The desired fuzzy set $A$ is defined as follows: put $A(k) = \alpha$, for all $k\in\omega$. Since $\alpha$ is $\Delta^0_2$, Proposition~\ref{remark:Delta_2} implies that the set $A$ is $\Delta^0_2$.
	
	Towards a contradiction, assume that $f(x,s)$ is a $\Delta^0_2$-approximation of $A$ such that the sequence $(m^f_{\Sigma}(0,s))_{s\in\omega}$ converges. There are two possible cases.
	
	\emph{Case~1.} Suppose that $\lim_s m^f_{\Sigma}(0,s) = 1$. Then choose a stage $s^{\ast}$ such that $m^f_{\Sigma}(0,s) = 1$ for all $s\geq s^{\ast}$. It is not hard to show that the set $\{ q\in\mathbb{Q} \,\colon q < \alpha\}$ is equal to 
	$\{ q\,\colon (\exists s \geq s^{\ast})[ q < f(0,s) ]\}$,
	and hence, this set is c.e. Then the real $\alpha$ is left-c.e., which gives a contradiction.

	\emph{Case~2.} Otherwise, $\lim_s m^f_{\Sigma}(0,s) = -1$. Then a similar argument shows that the real $\alpha$ is right-c.e.~--- again, a contradiction.
	
	We deduce that our fuzzy set $A$ has all desired properties.
	\qed
\end{proof}


\section{Concluding Remarks: Broadening the Fuzzy Ershov Hierarchy}

By way of conclusion, we briefly discuss two natural options for, first, \emph{refining} the Fuzzy Ershov Hierarchy, and, secondly, \emph{extending} the finite levels of the hierarchy. We leave many formal details to future work. In particular, even though we obtained a number of results concerning the new hierarchies discussed below, due to space constraints, we defer these results to an extended version of the present paper. 

\subsection{Counting Updates} 

We say that a $\Delta^0_2$-approximation of a fuzzy set $A$ ``has an update'' if $f(x,s+1)\neq f(x,s)$, for some $x,s\in\omega$. Observe that our notion of mind change, as in Definition~\ref{def:fuzzy-n-c.e.}, keeps track only of those updates which determine a change of monotonicity in the approximating function: e.g., if $f(x,s)$ is a $\Delta^0_2$-approxi\-ma\-ti\-on of a fuzzy set $A$, $m^f_{\Sigma}(x,s)=1$, and $f(x,s+1)>f(x,s)$, then $m^f_{\Sigma}(x,s+1)$ remains equal to $1$. So, one may explore what happens if one keeps track of \emph{all} updates for a given $\Delta^0_2$-approximation. To further motivate such approach, consider the example by Harkleroad~\cite{Harkleroad-84}:
\begin{example}\label{example_Harkleroad}
As usual, $K$ denotes the Halting problem. Define
\[
	H(x) = \begin{cases}
		1, & \text{if } x \in K,\\
		1/2, & \text{otherwise}.
	\end{cases}
\]
\end{example}

It is easy to see that $H$ is a fuzzy c.e. set. But note that, for any c.e.\ approximation $h$ of $H$ (recall that one assumes $h(x,0) = 0$ for all $x$), there must be an infinite crisp set $Z\subseteq \omega$ such that $h$ requires at least \emph{two} updates to approximate each $x\in Z$ (as otherwise,  $K=\{x \,\colon (\exists s)(h(x,s)=1) \}$ would be computable).

So, to distinguish $H$ from fuzzy c.e.\ sets which can be approximated  with at most one update, we propose the following new hierarchy:

\begin{definition} \label{def:inner_fuzzy_c_e}
	A fuzzy set $A$ is \emph{$[n]_1$-c.e.} if there is a computable function $f\colon \omega\times\omega \to [0,1]_{\mathbb{Q}}$ such that for all $x\in\omega$, we have:
	\begin{itemize}		
		\item $f(x,0) = 0$ and $\lim_{s} f(x,s)=A(x)$;
		
		\item $(\forall s)(f(x,s+1)\geq f(x,s))$;
		
		\item $|\{s\,\colon f(x,s+1)\neq f(x,s)\}| \leq n$.
	\end{itemize}
\end{definition}

For simplicity, the definition above is limited to the first level of the Fuzzy Ershov Hierarchy. But clearly, one could similarly stratify each class of fuzzy $n$-c.e.\ sets as follows. Intuitively, a set $A$ is fuzzy $[n_1,\ldots,n_m]_m$-c.e. if there is a $\Delta^0_2$-approximation $f$ to $A$ that, for each $x$, can go up at most $n_1$ times, and then down at most $n_2$ times, etc.~--- for $m$-many ups and downs. 

In a future work, we plan to carefully study all such refinements of the Fuzzy Ershov Hierarchy, together with their interplay.


\subsection{Going Transfinite}

In this paper, we talked only about the finite levels of the Fuzzy  Ershov Hierarchy. Similarly to the classical case, we could introduce transfinite levels. These levels are labelled by the notations of constructive ordinals, taken from Kleene's $\mathcal{O}$ (see \S\,11.7 in~\cite{Rogers-Book}).

\begin{definition}
Let $a$ be an element of $\mathcal{O}$. We say that a fuzzy $\Delta^0_2$ set $A$ \emph{belongs to the class $\Sigma^{-1}_a$} if there exist a $\Delta^0_2$-approximation $f(x,s)$ and a computable ``counting'' function $h\colon \omega^2\to \{ b\in \mathcal{O} \,\colon b<_{\mathcal{O}} a\}$ such that for all $x$ and $s$:
\begin{itemize}
	\item $\lim_s f(x,s) = A(x)$ and $f(x,0) = 0$;
	
	\item $h(x,s+1) \leq_{\mathcal{O}} h(x,s)$; and
	
	\item if $m^f_{\Sigma}(x,s+1) \neq m^f_{\Sigma}(x,s)$, then $h(x,s+1) \neq h(x,s)$.
\end{itemize}
\end{definition}

We note the following: if $a$ is the notation of a finite ordinal $n\geq 1$, then the $\Sigma^{-1}_a$ sets are \emph{precisely} the fuzzy $n$-c.e. sets.

In an extended version of this work, we will prove that, similarly to Proposition~\ref{prop:not-enough}, the classes $\Sigma^{-1}_a$, $a\in\mathcal{O}$, still \emph{do not exhaust} all fuzzy $\Delta^0_2$ sets. This provides a \emph{major difference} with the classical case: every crisp $\Delta^0_2$ set belongs to $\Sigma^{-1}_a$ for some $a\in\mathcal{O}$ (Theorem~6 in~\cite{Ershov-II}, see also Theorem~4.3 in~\cite{SYY-09}).

\subsubsection*{Acknowledgements.}
The authors are grateful to Professor Murat Ramazanov for his hospitality during their visit to Qaraghandy, Kazakhstan, in June 2019. Part of the research contained in the paper was carried out while Bazhenov, San Mauro, and Ospichev were visiting the Department of Mathematics of Nazarbayev University, Nur-Sultan.

%
%
%
\bibliographystyle{plain}
\bibliography{FuzzyReferences}

\begin{thebibliography}{10}

\bibitem{bedregal2006}
Benjamin Rene~Callejas Bedregal and Santiago Figueira.
\newblock Classical computability and fuzzy {T}uring machines.
\newblock In J.~R. Correa, A.~Hevia, and M.~Kiwi, editors, {\em LATIN 2006:
  Theoretical Informatics}, volume 3887 of {\em LNCS}, pages 154--165, Berlin,
  2006. Springer.

\bibitem{BG-89}
L.~Biacino and G.~Gerla.
\newblock Decidability, recursive enumerability and {K}leene hierarchy for
  {$L$}-subsets.
\newblock {\em Z. Math. Logik Grundlagen Math.}, 35(1):49--62, 1989.

\bibitem{DH-book}
R.~G. Downey and D.~R. Hirschfeldt.
\newblock {\em Algorithmic randomness and complexity}.
\newblock Springer, New York, 2010.

\bibitem{Ershov-I}
{Yu.}~L. Ershov.
\newblock A hierarchy of sets. {I}.
\newblock {\em Algebra Logic}, 7(1):25--43, 1968.

\bibitem{Ershov-II}
{Yu.}~L. Ershov.
\newblock On a hierarchy of sets, {II}.
\newblock {\em Algebra Logic}, 7(4):212--232, 1968.

\bibitem{Ershov-III}
{Yu.}~L. Ershov.
\newblock On a hierarchy of sets. {III}.
\newblock {\em Algebra Logic}, 9(1):20--31, 1970.

\bibitem{Gerla-book}
G.~Gerla.
\newblock {\em Fuzzy logic. Mathematical tools for approximate reasoning},
  volume~11 of {\em Trends in Logic}.
\newblock Kluwer Academic Publishers, Dordrecht, 2001.

\bibitem{Harkleroad-84}
L.~Harkleroad.
\newblock Fuzzy recursion, {RET}'s and isols.
\newblock {\em Z. Math. Logik Grundlagen Math.}, 30(26--29):425--436, 1984.

\bibitem{Harkleroad-88}
L.~Harkleroad.
\newblock Fuzzy regressivity and retraceability.
\newblock {\em Z. Math. Logik Grundlagen Math.}, 34(6):523--529, 1988.

\bibitem{Kleene-43}
S.~C. Kleene.
\newblock Recursive predicates and quantifiers.
\newblock {\em Trans. Am. Math. Soc.}, 53(1):41--73, 1943.

\bibitem{Gold}
E.~{Mark Gold}.
\newblock Limiting recursion.
\newblock {\em J. Symb. Log.}, 30(1):28--48, 1965.

\bibitem{Putnam}
H.~Putnam.
\newblock Trial and error predicates and the solution to a problem of
  {M}ostowski.
\newblock {\em J. Symb. Log.}, 30(1):49--57, 1965.

\bibitem{Rogers-Book}
H.~Rogers.
\newblock {\em Theory of recursive functions and effective computability}.
\newblock McGraw-Hill, New York, 1967.

\bibitem{Santos-70}
E.~S. Santos.
\newblock Fuzzy algorithms.
\newblock {\em Inf. Control}, 17(4):326--339, 1970.

\bibitem{Soare-16}
R.~I. Soare.
\newblock {\em Turing computability. Theory and applications}.
\newblock Springer, Berlin, 2016.

\bibitem{SYY-09}
F.~Stephan, Y.~Yang, and L.~Yu.
\newblock Turing degrees and the {E}rshov hierarchy.
\newblock In T.~Arai, J.~Brendle, H.~Kikyo, C.~T. Chong, R.~Downey, Q.~Feng,
  and H.~Ono, editors, {\em Proceedings of the 10th Asian Logic Conference},
  pages 300--321. World Scientific, Singapore, 2009.

\bibitem{wiedermann2002}
Jiri Wiedermann.
\newblock Fuzzy {T}uring machines revised.
\newblock {\em Comput. Inform.}, 21(3):251--264, 2002.

\bibitem{Zadeh-65}
L.~A. Zadeh.
\newblock Fuzzy sets.
\newblock {\em Inf. Control}, 8(3):338--353, 1965.

\bibitem{Zadeh-68}
L.~A. Zadeh.
\newblock Fuzzy algorithms.
\newblock {\em Inf. Control}, 12(2):94--102, 1968.

\end{thebibliography}

\end{document}